\documentclass[11pt]{article}

\usepackage{amsmath, amsthm, amsfonts, amssymb}
\usepackage[bookmarks]{hyperref}
\usepackage{indentfirst}
\usepackage{marvosym}

\usepackage{color}

\usepackage[a4paper]{geometry}
\newtheorem{lemma}{Lemma}[section]
\newtheorem{theorem}[lemma]{Theorem}
\newtheorem{proposition}[lemma]{Proposition}
\newtheorem{corollary}[lemma]{Corollary}
\renewenvironment{proof}[1][\noindent \proofname]{{\sc #1. }}{\qed}
\theoremstyle{definition}
\newtheorem{remark}[lemma]{Remark}
\newtheorem{example}[lemma]{Example}
\newtheorem{definition}[lemma]{Definition}
\newtheorem{notation}[lemma]{Notation}{\bf}{\rm}

\newcommand{\abs}[1]{\ensuremath{\left| #1 \right|}}
\newcommand{\op}{\operatorname}
\newcommand{\ce}[2]{\pmb{\op{C}}_{#1}(#2)}
\newcommand{\ze}[1]{\pmb{\op{Z}}(#1)}
\newcommand{\irr}[1]{{\op{Irr}}(#1)}
\newcommand{\Min}[1]{\op{Min}_G(#1)}

\begin{document}

\title{\bf On $G$-character tables for normal subgroups}

\author{\sc  M.J. Felipe $^{*}$ $\cdot$ M.D. P\'erez-Ramos $^{\diamond}$ $\cdot$ V. Sotomayor
\thanks{Instituto Universitario de Matemática Pura y Aplicada (IUMPA-UPV), Universitat Polit\`ecnica de Val\`encia, Camino de Vera s/n, 46022 Valencia, Spain. \Letter: \texttt{mfelipe@mat.upv.es, vicorso@doctor.upv.es}\newline \indent $^{\diamond}$Departament de Matemàtiques, Universitat de València, C/ Doctor Moliner 50, 46100 Burjassot (València), Spain. \Letter: \texttt{dolores.perez@uv.es}
\newline \rule{6cm}{0.1mm}\newline
The authors are supported by Proyecto PGC2018-096872-B-I00 (MCIU/AEI/FEDER, UE), and by Proyectos PROMETEO/2017/057 and AICO/2020/298, Generalitat Valenciana (Spain). \newline
}}

\date{}

\maketitle

\begin{abstract}
\noindent Let $N$ be a normal subgroup of a finite group $G$. From a result due to Brauer, it can be  derived that the character table of $G$ contains square submatrices which are induced by the $G$-conjugacy classes of elements in $N$ and the $G$-orbits of irreducible characters of $N$. In the present paper, we provide an alternative approach to this fact through the structure
of the group algebra. We also show that such matrices are non-singular and become a useful tool to obtain information of $N$ from the character table of $G$.

\medskip


\smallskip

\end{abstract}


\section{Introduction}

In the sequel, all groups considered are finite. In general, if $N$ is a normal subgroup of a group $G$,  the character table of $N$ can not be computed from the one of $G$. However, normal subgroups are clearly detected from the character table of the group, as intersections of kernels of irreducible characters. If we consider the action by conjugation of $G$ on $N$, then the orbits are the conjugacy classes of $G$ which are contained in $N$, i.e. conjugacy classes $x^G=\{x^g\mid g\in G\}$ with $x\in N$, the so-called $G$\emph{-conjugacy classes} of $N$. Since $N$ is the union of those $G$-conjugacy classes, it is natural to wonder whether those columns of the character table of $G$ provide structural information of $N$. In this setting, for instance, recent investigations show that the sizes of the $G$-conjugacy classes are related to the structure of $N$, though one easily checks   that the prime divisors of the sizes of the $G$-conjugacy classes may even not divide the order of $N$. We refer to \cite{BFM} for a survey about the influence of conjugacy  classes contained in normal subgroups on the normal structure of the group, mainly focusing on the framework of graphs associated to the conjugacy classes. Further, the authors of \cite{FGS} analysed normal subgroups $N$ of a group $G$ that contain (non-)vanishing $G$-conjugacy classes. This turns out to be closely related to the research presented  in this paper (see Remark~\ref{zeros}).

Conjugation of $G$ on $N$ clearly induces an action of $G$ on the set of conjugacy classes of $N$,  and $G$-conjugacy classes are in one-to-one correspondence with the orbits of that action, since each $G$-conjugacy class appears as the union of the elements of one of those orbits. More precisely, if $x\in N$, then $x^G=\bigcup_{i=1}^r(x^N)^{h_i}=\bigcup_{i=1}^r (x^{h_i})^N,$ for suitable elements $h_i\in G$, $1\le i\le r$, where we may assume $h_1=1$; in particular, $|x^G|= r|x^N|$ for some integer $r\ge 1$. Clifford's theorem (c.f. \cite[Theorems~(6.2),~(6.5)]{ISA}  for the character and module theoretic settings, respectively) may be seen as a character version of this easy fact on conjugacy classes, as follows.

The group $G$ acts by conjugation on the set $\irr{N}$ of irreducible complex characters of $N$. We recall that  if $\theta$ is a class function of $N$ and $g\in G$, then $\theta^g$ is the  {\it class function  conjugate} to $\theta$, defined as $\theta^g(x)=\theta(gxg^{-1})$ for all $x\in N$. If $\theta\in \irr{N}$, then
$\theta^g\in \irr{N}$ and $\widehat{\theta}=\sum_{i=1}^t \theta^{g_i}$ is a {\it minimal $G$-invariant character} of $N$ (see Sections~\ref{G-tables}~and~\ref{sourceG-tables}), where $\{\theta^{g_i}\mid i=1,\dots, t\}=\{\theta^{g}\mid g\in G\}$ is the orbit of $\theta$ under the action by conjugation of $G$ on $\irr{N}$, being the set $\{g_i\in G\mid i=1,\dots, t\}$ a right transversal in $G$ of $I_G(\theta)=\{g\in G\mid \theta^g=\theta\}$, the {\it inertia subgroup} of $\theta$ in $G$. Clifford's theorem (see \cite[Theorem~(6.2)]{ISA}) states that irreducible characters $\chi$ of $G$ whose restrictions $\chi_{N}$ to $N$ have $\theta$ as constituent satisfy $\chi_N= e\,\widehat{\theta}$ for suitable integers $e$, known as {\it ramification numbers}. This fact introduces an equivalence relation (respect to $N$) on the set $\irr{G}$ of irreducible complex characters of $G$, being two elements equivalent if their restrictions to the normal subgroup $N$ have a common irreducible constituent (Definition~\ref{relEqIrr(G)}). In this case, the equivalence  classes of irreducible characters of $G$ are in one-to-one correspondence with the orbits of irreducible characters of $N$ under the action of $G$ by conjugation, by Clifford's theorem.

 Now, from the character table of $G$, if we focus on the $G$-conjugacy classes that make up $N$, then by taking one character in each of those equivalence classes, we obtain a submatrix which is associated to the normal subgroup $N$. We call these submatrices the $G$\emph{-character tables of }$N$ (Definition~\ref{G-character table}), and they happen to be square matrices, since the numbers of orbits in the actions of $G$ on the irreducible characters and conjugacy classes of $N$ are equal, by  a result due to Brauer (Theorem~\ref{brauer_theorem}, Corollary~\ref{corolario Brauer}).

In the present paper, we provide an alternative proof of this last fact in a module theoretic setting, analysing the structure of the subalgebra $\ze{\mathbb{K}[G]}\cap \mathbb{K}[N]$ of the group algebra $\mathbb{K}[G]$ over a splitting field $\mathbb{K}$ for $G$ (in particular, if $\mathbb{K}$ is algebraically closed) with characteristic not dividing the order of $G$. It is well-known that  the formal sums of the elements in the conjugacy classes of $G$ form a basis for $\ze{\mathbb{K}[G]}$, the center of the group algebra of $G$ over $\mathbb{K}$, whose dimension then coincides with the number of conjugacy classes of $G$ as well as with the number of pairwise non-isomorphic irreducible $\mathbb{K}[G]$-modules. Now it turns out that  the formal sums of the elements in the $G$-conjugacy classes that built $N$ form a basis for $\ze{\mathbb{K}[G]}\cap \mathbb{K}[N]$, and it is also proven that the dimension of this algebra coincides with the number of {\it minimal $G$-invariant $\mathbb{K}[N]$-modules}, up to isomorphism, which is equal to the number of orbits in the action by conjugation of $G$ on the set of irreducible $\mathbb{K}[N]$-modules, by Clifford's theorem. (See Sections~\ref{section 2}~and~\ref{section 3}, Lemma~\ref{lemma_classes} and Theorem~\ref{main_theorem}.) In particular, when $\mathbb{K}=\mathbb{C}$ is the complex field, it follows that the $G$-character tables of $N$ are square, as mentioned.

We show then in Section~\ref{G-tables} how the character table of $G$ nicely displays the set of irreducible characters of $G$ with a common irreducible constituent when restricting to the normal subgroup $N$, as well as the relations between their ramifications numbers (Corollary~\ref{sizes}). Also as applications, we prove that the $G$-character tables are non-singular, and we obtain some arithmetical relations among significant integers associated to the normal subgroup $N$ as, for instance, its character degrees, the indices in $G$ of the inertia subgroups of its irreducible characters and the ramification numbers, and these relations can be read off a $G$-character table of $N$ (Theorem~\ref{square}, Corollary~\ref{relations}).  In particular, it will be possible to know from the character table of $G$ whether $\chi_N\in \irr{N}$ if $\chi\in \irr{G}$ (Corollary~\ref{restriction}), and we will be able to compute  the exact values of those parameters in some occasions, as for instance for normal Hall subgroups (Remark~\ref{remarks}). Also we show that the prime divisors of the degrees of the minimal $G$-invariant characters of a normal subgroup   are known from the character table of the group  (Corollary~\ref{primesinvariants}).

In Section~\ref{sourceG-tables}, the minimal $G$-invariant characters of $N$ are proven to form a basis of the $\mathbb{C}$-vector space  of the $G$-invariant class functions of $N$ (Proposition~\ref{G-classfunctions}), and minimal $G$-invariant characters become relevant to $G$-character tables by playing the role that irreducible characters do to the character table of the group. They can be used as row indexes of the so called  {\it  $G$-invariant table} of a normal subgroup (Definition~\ref{sourcetable}), which is proven to be a useful tool. It is first applied to give an extension of the aforementioned Brauer's result (Theorem~\ref{Brauer}). As a consequence we obtain that the number of rows of the $G$-character tables that are real valued (i.e. the number of real valued minimal $G$-invariant characters) coincides with the number of $G$-conjugacy classes of $N$ that are \emph{real} (i.e. the $G$-conjugacy classes $x^G$ such that $x^G=(x^{-1})^G$, where $x\in N$) (Corollary~\ref{CorBrauer}). Indeed, we show that real G-conjugacy classes of normal subgroups can provide structural information about the subgroup, consistent with previous results in the literature about the relation between the number of real conjugacy classes and the structure of groups (Lemma~\ref{odd_order}, Corollary~\ref{teo_iwa}). Finally, in Section~\ref{Section examples}, some examples are provided  which show the scope of the results presented.

To close with, we mention that the notation and terminology used are standard within theories of groups and representations, and they are taken mainly from the books \cite{ISA}, \cite{Hu} and \cite{DH}.

\section{Preliminaries}\label{section 2}

We gather some notation, basic concepts and facts about algebras and modules which are used in the paper. In the following, $\mathbb{K}$ always denote a field and, for any group $X$, the \emph{group algebra over $\mathbb{K}$} is denoted by $\mathbb{K}[X]$. In general, $\mathbb{K}$-algebras $A$ are considered with identity $1_A$, and both them and their modules have  finite dimension as $\mathbb{K}$-vector spaces. If $A$ is a $\mathbb{K}$-algebra,  $V$ is an $A$-module and $M$ is an irreducible $A$-module, the \emph{$M$-homogeneous component} of $V$ is $H_V(M)= \sum \{X \mid X \text{ is an $A$-submodule of }V,\ V\cong M\}$ the sum of all those submodules of $V$ which are isomorphic to $M$. An $A$-module $V$ is understood to be \emph{completely reducible} if it is a sum of irreducible $A$-submodules or, equivalently, if $V=V_1\oplus \cdots \oplus V_r$ is a direct sum of irreducible $A$-submodules $V_1, \ldots, V_r$. Note that, in this case, $H_V(M)=\bigoplus\{V_i \mid V_i\cong M\}$; in particular, every completely reducible $A$-module is a direct sum of its homogeneous components.

First we establish a version for modules of the facts about characters mentioned in the Introduction. Let $N$ be a normal subgroup of a group $G$, and let $W$ be a $\mathbb{K}[N]$-module. Let $\overline{W}$ be a copy of the $\mathbb{K}$-vector space $W$, and let $\overline{w}$ denote the image of each $w\in W$ under some $\mathbb{K}$-linear isomorphism from $W$ to $\overline{W}$. For a fixed element $g\in G$ define an action of $N$ on $\overline{W}$ by $$\overline{w}\,n=\overline{w\,(gng^{-1})}$$ for all $w\in W$ and $n\in N$. Then $\overline{W}$ becomes a $\mathbb{K}[N]$-module under this action, called the \emph{conjugate module} of $W$ by $g$ and denoted by $W^g$ (\cite[B. Definition (7.2)]{DH}). 

It is worth noticing that if $W$ is a $\mathbb{K}[N]$-submodule of a $\mathbb{K}[G]$-module $V$, then $Wg=\{wg\mid w\in W\}\subseteq V$ is a $\mathbb{K}[N]$-submodule of $V$ isomorphic to $W^g$. 

The \emph{inertia subgroup} of $W$ in $G$ is the subgroup of $G$ defined as $I_G(W)=\{g\in G\mid W^g\cong W \text{ as } \mathbb{K}[N]\text{-modules}\}.$
The module $W$ is said to be \emph{$G$-invariant} if $W^g$ is isomorphic to $W$ for all $g\in G$, i.e. if $I_G(W)=G$. If $W$ is an irreducible $\mathbb{K}[N]$-module and $g\in G$, then $W^g$ is also irreducible, and so conjugation defines an action of the group $G$ on the set of irreducible $\mathbb{K}[N]$-modules. The orbit of $W$, up to isomorphism of $\mathbb{K}[N]$-modules, is the set $\{W^{g_i}\mid i=1,\dots,t\}$ where the set $\{g_i\in G\mid i=1,\dots, t\}$ is a right transversal in $G$ of $I_G(W)$. The following result is easily checked.

\begin{lemma}\label{conjugate} With the notation above, if $W$ is an irreducible $\mathbb{K}[N]$-module, the following assertions hold:
\begin{enumerate}\item The module $\widehat{W}:=\bigoplus_{i=1}^tW^{g_i}$ is a minimal $G$-invariant $\mathbb{K}[N]$-module.
\item Every minimal $G$-invariant $\mathbb{K}[N]$-module can be constructed from one orbit of irreducible modules in this way, and this defines a one-to-one correspondence between the set of orbits in the action by conjugation of $G$ on the set of irreducible $\mathbb{K}[N]$-modules and the set of minimal $G$-invariant $\mathbb{K}[N]$-modules (up to isomorphism).
\item Every completely reducible $G$-invariant $\mathbb{K}[N]$-module is a direct sum of  minimal $G$-invariant $\mathbb{K}[N]$-modules.
\end{enumerate}
\end{lemma}

From Clifford's theorem (c.f. \cite[Lemma (6.4), Theorem (6.5)]{ISA}) the following connection with irreducible $\mathbb{K}[G]$-modules is easily established. For notation, if $e$ is a positive integer and $V$ is any module, $e\,V=V\oplus\overset{e}{ \cdots}\oplus V$ denotes the direct sum of $e$ copies of $V$.

\begin{lemma}\label{Irr(G)} Let $N$ be a normal subgroup of a group $G$. For irreducible $\mathbb{K}[G]$-modules $V$ and $U$ the following statements are equivalent:
\begin{itemize}\item[(i)] The restricted $\mathbb{K}[N]$-modules $V_N$ and $U_N$, i.e. the modules $V$ and $U$ viewed as $\mathbb{K}[N]$-modules, have a common irreducible $\mathbb{K}[N]$-submodule (up to isomorphism).
\item[(ii)] There exist an irreducible $\mathbb{K}[N]$-submodule $W$ and positive integers $e_V$ and $e_U$ such that $V_N\cong e_V\, \widehat{W}$ and $U_N\cong e_U\, \widehat{W}$.
\item[(iii)] There are positive integers $e_V$ and $e_U$ such that $e_U\, V_N\cong e_V\, U_N$.
\end{itemize}
\end{lemma}

The previous lemma allows us to introduce an equivalence relation on the set of irreducible $\mathbb{K}[G]$-modules as follows:

\begin{definition}\label{relEqMod(G)}  Let $N$ be a normal subgroup of a group $G$. Two irreducible $\mathbb{K}[G]$-modules $V$ and $U$ are defined to be {\it equivalent respect to $N$} if they satisfy any of the conditions in Lemma~\ref{Irr(G)}.
\end{definition}
It is worth emphasizing the following fact: The equivalence class of each irreducible $\mathbb{K}[G]$-module $V$  is associated to a  minimal $G$-invariant $\mathbb{K}[N]$-module $\widehat{W}$, with $W$ an irreducible $\mathbb{K}[N]$-submodule of $V_N$, and this defines  a one-to-one correspondence between the set of equivalence classes of irreducible $\mathbb{K}[G]$-modules and the set of minimal $G$-invariant $\mathbb{K}[N]$-modules (up to isomorphism).

We will make use of the structure of semisimple algebras that we gather in the next result, which includes Wedderburn's theorem. An algebra $A$ is understood to be \emph{semisimple} if the regular $A$-module $A$, i.e. $A$ itself viewed as $A$-module under right multiplication, is completely reducible. If $M$ is an irreducible $A$-module, then $H_A(M)$ is the $M$-homogeneous component of $A$ as regular module. We refer to \cite[Theorem (1.15) and Corollary (1.17)]{ISA} and to \cite[V. Satz (3.8), Hauptsazt (4.4), Satz (4.5)]{Hu} for details.

\begin{theorem}
\label{wedderburn}
Let $A$ be a semisimple $\mathbb{K}$-algebra. Then:
\begin{itemize}
\setlength{\itemsep}{-1mm}
\item[\emph{(a)}] There is a positive integer $h$ and minimal bilateral ideals $A_1,\dots,A_h$, such that $A=\bigoplus_{i=1}^hA_i$. For $i\neq j$ it holds that $A_iA_j=0$. Moreover, each bilateral ideal of $A$ is a direct sum of some of $A_1,\dots, A_h$.
\item[\emph{(b)}] There exist exactly $h$ pairwise non-isomorphic irreducible $A$-modules $V_1,\dots, V_h$, and w.l.o.g. $A_i=H_A(V_i)$ for each $i=1,\dots, h$. If $V$ is an irreducible $A$-module, then $VA_i=0$ for $i=1,\dots, h$, unless $V\cong V_i$.
\item[\emph{(c)}] Each $A_i$ is a $\mathbb{K}$-algebra isomorphic to $\op{End}_{D_i}(V_i)$, the algebra of $D_i$-endomorphisms of $V_i$, for the division algebra $D_i=\op{End}_A(V_i)$.
\item[\emph{(d)}] Each $A_i$ is an algebra with identity element $1_{A_i}=e^i$, being $1_A=e^{1}+\cdots+e^{h}$ with $e^{i}\in A_i$ for each $i=1,\dots, h$.
\end{itemize}
Whenever $D_i\cong \mathbb{K}$, for some $i\in \{1,\dots,h\}$,  then  $A_i\cong \op{End}_{\mathbb{K}}(V_i)$, and $\ze{A_i}\cong \mathbb{K}$, being $\ze{A_i}=\{x\in A_i\mid xy=yx\ \text{for all }y\in A_i\}$ the center of the algebra $A_i$. This holds, in particular, if $\mathbb{K}$ is  algebraically closed.
\end{theorem}

Finally, we state Brauer's result cited in the Introduction.

\begin{theorem}\emph{\cite[Theorem~(6.32),  Corollary (6.33)]{ISA}}
\label{brauer_theorem}
Let $A$ be a group that acts on $\irr{G}$ and on the set of conjugacy classes of a group $G$. Assume that $\chi^a(g^a)=\chi(g)$ for all $\chi\in\irr{G}, g\in G$ and $a\in A$, where $g^a$ is an element of the conjugacy class $(g^G)^a$. Then for each $a\in A$, the number of fixed irreducible characters of $G$ is equal to the number  of fixed conjugacy classes.

Consequently, the numbers of orbits in the actions of $A$ on $\irr{G}$ and conjugacy classes of $G$ are equal.
\end{theorem}

\begin{corollary}\label{corolario Brauer} If $N$ is a normal subgroup of a group $G$, then the numbers of orbits in the actions by conjugation of $G$ on $\irr{N}$ and the set of conjugacy classes of $N$ are equal.
\end{corollary}

\begin{proof} Notice that, by  the definition of conjugate character, the considered actions satisfy the required condition in Theorem~\ref{brauer_theorem}.
\end{proof}

\begin{remark} According to Brauer's theorem, if $N$ is a normal subgroup of a group $G$, and we consider the actions by conjugation of $G$ on $\irr{N}$ and the set of conjugacy classes of $N$, then for each $g\in G$, the number of fixed irreducible characters of $N$ is equal to the number  of fixed conjugacy classes of $N$, but it is not difficult to find examples showing that the number of conjugacy classes of $N$ that are invariant under the action of $G$ may not coincide with the number of irreducible characters of $N$ that are $G$-invariant. For instance, let $N=\langle a\rangle \times \langle b \rangle \times \langle c \rangle $ be a $2$-elementary abelian group. Consider the action of a cyclic group $H=\langle x\rangle$ of order $4$ on $N$, in such way that $a^x=ab, b^x=bc, c^x=c$. Let $G=N \rtimes H$ be the corresponding semidirect product. One can check, for instance by using GAP, that the mentioned numbers are not equal for the normal subgroup $N$.

\end{remark}

\section{The algebra $\ze{\mathbb{K}[G]}\cap \mathbb{K}[N]$}\label{section 3}

For a conjugacy class $K$ in a group $G$, let us denote its formal sum in the group algebra by $\widehat{K}=\sum_{x\in K} x\in\mathbb{K}[G]$. It is a fact of common knowledge that the algebra $\ze{\mathbb{K}[G]}$ is generated by the set of formal sums of the conjugacy classes of $G$. Further, its dimension is equal to the number of non-isomorphic irreducible $\mathbb{K}[G]$-modules, if in addition $\mathbb {K}$ is  a splitting field for the group $G$, i.e. $\op{End}_{\mathbb {K}[G]}(V)\cong \mathbb{K}$ for any irreducible $\mathbb{K}[G]$-module $V$, in particular, if $\mathbb{K}$ is algebraically closed, and the characteristic of $\mathbb{K}$ does not divide the order of $G$.

Our main goal in this section is to carry out an analogous study for the subalgebra $\ze{\mathbb{K}[G]}\cap \mathbb{K}[N]$ for a normal subgroup $N$ of $G$. We split our development into the following two results.

\begin{lemma}
\label{lemma_classes}
Let $N$ be a normal subgroup of a group $G$, and $\mathbb{K}$ be a field. Let $\{K_1,\ldots, K_l\}$ be the set of $G$-conjugacy classes of $N$, and let
$\widehat{K_i}$ be the formal sum of $K_i$ in $\mathbb{K}[G]$, for each $i=1,\dots, l$. Then $\{\widehat{K_1},\ldots,\widehat{K_l}\}$ forms a basis of the $\mathbb{K}$-algebra $\ze{\mathbb{K}[G]}\cap \mathbb{K}[N]$.
\end{lemma}

\begin{proof}
Let $\{K_1, \ldots, K_l, \ldots, K_h\}$ be the set of conjugacy classes of $G$, so $K_i\subseteq N$ precisely when $1\le i\le l
$. We know that $\{\widehat{K_1}, \ldots, \widehat{K_l}\}\subseteq \ze{\mathbb{K}[G]}\cap \mathbb{K}[N]$ is a $\mathbb{K}$-linear independent set. If $x \in \ze{\mathbb{K}[G]}\cap \mathbb{K}[N] \subseteq \ze{\mathbb{K}[G]}$, then \[x = \alpha_1\widehat{K_1} + \cdots + \alpha_h\widehat{K_h},\]
for suitable $\alpha_1, \ldots, \alpha_h\in \mathbb{K}$. Moreover, $x\in \mathbb{K}[N]$ can be uniquely written as $x=\sum_{n\in N}\beta_n\, n$ with $\beta_n\in\mathbb{K}$. Since $\widehat{K_i}\subseteq N$ if $1\leq i \leq l$, we deduce that $\alpha_i=0$ whenever $l+1\leq i \leq h$, so $\{\widehat{K_1},\ldots, \widehat{K_l}\}$ is a basis of $\ze{\mathbb{K}[G]}\cap \mathbb{K}[N]$.
\end{proof}
\medskip

Next we introduce some notation, which will be used in the statement of our main result, Theorem~\ref{main_theorem}.

\begin{notation}
\label{notation} Let $\mathbb {K}$ denote a splitting field for a group $G$, whose characteristic does not divide the order of $G$, which implies by Maschke's theorem that the group algebra $\mathbb {K}[G]$ is semisimple.  Hence Theorem~\ref{wedderburn} applies for the group algebra $\mathbb {K}[G]$.

Since the number of non-isomorphic irreducible $\mathbb{K}[G]$-modules is finite, let the positive integer $k$ denote the number of equivalence classes in the equivalence relation respect to a normal subgroup $N$, on the set of irreducible $\mathbb {K}[G]$-modules (Definition~\ref{relEqMod(G)}). For each $i=1,\dots, k$, let $\{V_{i1}, \ldots, V_{i{s_i}}\}$, for some positive integer $s_i$, denote a system of pairwise non-isomorphic irreducible $\mathbb{K}[G]$-modules in the same equivalence class, and set
$$T_i:=H_{\mathbb{K}[G]}(V_{i1})\oplus \cdots \oplus H_{\mathbb{K}[G]}(V_{i{s_i}}),\ \text{and } L_i:=T_i\cap \mathbb{K}[N],$$
which are bilateral ideals of $\mathbb{K}[G]$  and $\mathbb{K}[N]$, respectively.

\end{notation}


\begin{theorem}
\label{main_theorem}
Let $N$ be a normal subgroup of a group $G$. With  Notation~\ref{notation}, the following statements hold:
\begin{enumerate}
\item[\emph{(a)}] $\mathbb{K}[N] = L_1 \oplus \cdots \oplus L_k$.
\item[\emph{(b)}] $\ze{\mathbb{K}[G]}\cap \mathbb{K}[N] = (L_1\cap \ze{T_1}) \oplus \cdots \oplus (L_k\cap \ze{T_k})$.
\item[\emph{(c)}] $\emph{dim}_{\mathbb{K}}(L_i\cap \ze{T_i})=1$, for every $i=1,\dots,k$.
\item[\emph{(d)}] $\emph{dim}_{\mathbb{K}}(\ze{\mathbb{K}[G]}\cap \mathbb{K}[N])=k$, which is equal to the number of non-isomorphic minimal $G$-invariant $\mathbb{K}[N]$-modules, as well as to the number of orbits in the action by conjugation of $G$ on the set of irreducible $\mathbb{K}[N]$-modules.
\end{enumerate}
\end{theorem}

\begin{proof} Following  Notation~\ref{notation}, let us consider $$\mathbb{K}[G]=T_1\oplus \cdots \oplus T_i\oplus\cdots \oplus T_k= \bigoplus_{i=1}^k \big(H_{\mathbb{K}[G]}(V_{i1})\oplus \cdots \oplus H_{\mathbb{K}[G]}(V_{i{s_i}})\big).$$
According to Lemmas~\ref{Irr(G)}~and~\ref{conjugate},  for each $i=1,\dots, k$, and each $j=1,\dots, s_i$, we have that
$$(V_{ij})_N \cong e_{ij}\widehat{W}_{i1} = e_{ij}(W_{i1} \oplus \cdots \oplus W_{i{t_i}}),$$
for suitable positive integers  $e_{ij}$ and $t_i$, and being  $\widehat{W}_{i1}$ and $\{W_{i1}, \dots, W_{i{t_i}}\}$, respectively, the corresponding minimal $G$-invariant $\mathbb{K}[N]$-module and orbit in the action of $G$ by conjugation on the set of irreducible $\mathbb{K}[N]$-modules,  up to isomorphism of $\mathbb{K}[N]$-modules.

By the generalization for modules of the Frobenius reciprocity theorem, due to Nakayama, it follows that $$\bigcup_{i=1}^k \{W_{i1}, \dots, W_{i{t_i}}\}$$ is a complete system of pairwise non-isomorphic irreducible $\mathbb{K}[N]$-modules. (See~\cite[B. Theorem~6.11]{DH}.) Therefore, $$\mathbb{K}[N]= A_1\oplus \cdots \oplus A_i\oplus\cdots \oplus A_k=\bigoplus_{i=1}^k \big(H_{\mathbb{K}[N]}(W_{i1}) \oplus \cdots \oplus  H_{\mathbb{K}[N]}(W_{i{t_i}})\big),$$
being $A_i:=H_{\mathbb{K}[N]}(W_{i1}) \oplus \cdots \oplus  H_{\mathbb{K}[N]}(W_{i{t_i}})$ bilateral ideal of $\mathbb{K}[N]$,  for each $i \in \{1, \dots, k\}$.

Let $i\in \{1,\dots, k\}$. We claim that $A_i= L_i:=T_i\cap \mathbb{K}[N]$. Let $C:=\mathbb{K}[G]_N$ (that is, $C$ denotes $\mathbb{K}[G]$ viewed as  $\mathbb{K}[N]$-module). Then $$C= R_1\oplus \cdots \oplus R_i\oplus\cdots \oplus R_k=\bigoplus_{i=1}^k \big(H_{C}(W_{i1}) \oplus \cdots \oplus  H_{C}(W_{i{t_i}})\big),$$
being $R_i:=H_{C}(W_{i1}) \oplus \cdots \oplus  H_{C}(W_{i{t_i}})$, which is a $\mathbb{K}[N]$-submodule of $C$,  for each $i \in \{1, \dots, k\}$. Since $\mathbb{K}[N]\subseteq C$, it is clear that $A_i\subseteq R_i$. Besides, $T_i\subseteq R_i$. Hence, by Dedekind's identity, we have
$$R_i = R_i \cap \mathbb{K}[G] = R_i \cap (T_1 \oplus \cdots \oplus T_k) = T_i \oplus (R_i \cap \bigoplus_{j \neq i} T_j).$$

Let us prove that $S_i:=R_i \cap \bigoplus_{j \neq i} T_j=0$. Arguing by contradiction, suppose that $S_i \neq 0$, and let $W$ be an irreducible $\mathbb{K}[N]$-submodule of $S_i$. Then $W \subseteq R_i$, so $W \cong W_{is}$ for some $s \in \{1, \dots , t_i\}$. Since additionally $W \subseteq\bigoplus_{j \neq i} T_j$, we have $W \cong W_{jr}$, for some $j \neq i$ and $r \in \{1, \dots, t_j\}$, which is a contradiction. Thus $S_i=0$ and $A_i\subseteq R_i=T_i$. It can analogously be proven that $T_i \cap \bigoplus_{j\neq i} A_j=0$, so that
 $$L_i=T_i\cap \mathbb{K}[N]= T_i \cap (A_1\oplus \cdots \oplus A_k) = A_i \oplus (T_i \cap \bigoplus_{j\neq i} A_j) = A_i,$$ which proves the claim. Since $\mathbb{K}[N]= A_1\oplus \cdots \oplus A_i\oplus\cdots \oplus A_k$, Part (a) follows.

In order to prove part (b), we notice that $T_iT_j=0$, whenever $i\neq j$, from Theorem~\ref{wedderburn}(a). Then, since $\mathbb{K}[G]=T_1\oplus\cdots \oplus T_k$, it follows that $\ze{\mathbb{K}[G]}=\ze{T_1}\oplus\cdots\oplus\ze{T_k}$, and so
\begin{eqnarray*}
\mathbb{K}[N] \cap \ze{\mathbb{K}[G]}&=&(L_1 \oplus \cdots \oplus L_k) \cap (\ze{T_1} \oplus \cdots \oplus \ze{T_k}) = \\
& = & \bigoplus_{i=1}^k L_i \cap \ze{T_i}=\bigoplus_{i=1}^k \ce{L_i}{T_i},
\end{eqnarray*}
being $\ce{L_i}{T_i}=\{x\in L_i\mid xy=yx\ \text{for all } y\in T_i\}$, and where the second equality follows from the uniqueness of the decomposition of the elements within the direct sum $T_1\oplus \cdots \oplus T_k$, as  $L_i\subseteq T_i$ for each $i=1,\dots, k$. Part (b) is now proven.

Part~(d) is clearly a consequence of Part~(c) together with Lemmas~\ref{Irr(G)}~and~\ref{conjugate}. So we finally prove Part~(c). Let $i\in \{1,\dots, k\}$. We aim to prove that  $\text{dim}_{\mathbb{K}}(\ce{L_i}{T_i}) = 1$.

We claim first that $\ce{L_i}{T_i}\neq 0$. Observe that $\mathbb{K}[G]$ and $\mathbb{K}[N]$ are two $\mathbb{K}$-algebras with identity element $1=1_G=1_{\mathbb{K}[G]}=1_{\mathbb{K}[N]}$. Moreover,
$1 = 1_{L_1} + \cdots + 1_{L_{k}}=1_{T_1} + \cdots + 1_{T_{k}}$
being $1_{L_j}\in L_j$ and $1_{T_j}\in T_j$  the identity elements of the $\mathbb{K}$-algebras $L_j$ and $T_j$, respectively, for each for $j \in \{1, \dots, k\}$, by Theorem~\ref{wedderburn}(a). Since $1_{L_i} \in L_i \subseteq T_i$, it follows that $1_{L_i}=1_{T_i} \in \ce{L_i}{T_i} \neq 0$, which proves the claim.

Since $T_i:=H_{\mathbb{K}[G]}(V_{i1})\oplus \cdots \oplus H_{\mathbb{K}[G]}(V_{i{s_i}})$, we derive from Theorem~\ref{wedderburn}(a),(c), and the fact that $\mathbb{K}$ is a splitting field for $G$, that $\ze{H_{\mathbb{K}[G]}(V_{ij})}\cong \mathbb{K}$ for all $j=1,\dots,s_i$, and
$$\ce{L_i}{T_i}\subseteq \ze{T_i}=\bigoplus_{j=1}^{s_i}\ze{H_{\mathbb{K}[G]}(V_{ij})}\cong \mathbb{K}\oplus \overset{s_i}{ \cdots}\oplus \mathbb{K}.$$
For each $j=1,\dots,s_i$, set $Z_{ij}:=\ze{H_{\mathbb{K}[G]}(V_{ij})}$. Let
$$\op{in}: \ce{L_i}{T_i} \longrightarrow \bigoplus_{j=1}^{s_i} Z_{ij} \qquad\qquad \pi_{r}: \bigoplus_{j=1}^{s_i} Z_{ij} \longrightarrow Z_{ir} $$ be the natural injection and projection of the direct sum, respectively, for each $r \in \{1, \dots, s_i\}$, i.e. $\op{in}(x)=x$ for every $x\in \ce{L_i}{T_i}$, and $\pi_{r}(\sum_{j=1}^{s_i}x_j)=x_r$ for every $x_j\in Z_{ij}$, $1\le j\le s_i$. Let us consider the induced projection $\pi_r \circ \op{in}:  \ce{L_i}{T_i} \rightarrow Z_{ir}$, for each $r=1,\dots,s_i$, which is a $\mathbb{K}$-linear map.

If $s_i=1$, it is clear that $\text{dim}_{\mathbb{K}}(\ce{L_i}{T_i}) = 1$, and we are done. Hence, we may assume that $s_i\ge 2$. We show next that $\ce{L_i}{T_i}$ is a subdirect product of $\displaystyle\bigoplus_{j=1}^{s_i} Z_{ij}$, i.e. $(\pi_r \circ \op{in}) (\ce{L_i}{T_i})=Z_{ir}$ for all $r=1,\dots,s_i$.  In order to show that $\pi_r \circ \op{in}$  is onto, let us suppose by contradiction that $(\pi_r \circ \op{in}) (\ce{L_i}{T_i})=0$, for some $r\in\{1,\dots,s_i\}$. Hence
$$\ce{L_i}{T_i} \subseteq \mathbb{K}[N] \cap \left(\bigoplus_{\underset{j\neq r}{j=1}}^{s_i} Z_{ij}\right) \subseteq \mathbb{K}[N] \cap \left(\bigoplus_{\underset{j\neq r}{j=1}}^{s_i} H_{\mathbb{K}[G]}(V_{ij})\right).$$ We claim that this last intersection is trivial, which will imply  the contradiction $\ce{L_i}{T_i}=0$. More generally, we are going to prove that for every $I \subsetneq \{1, \dots, s_i\}$ it holds
$$C_i^I := \mathbb{K}[N] \cap \left(\bigoplus_{j \in I \subsetneq \{1, \dots, s_i\}}H_{\mathbb{K}[G]}(V_{ij})\right)=0.$$
Suppose that $C_i^I \neq 0$ for some $I \subsetneq \{1, \dots, s_i\}$. We notice  that $C_i^I$ is a bilateral  ideal of $\mathbb{K}[N]$ by Theorem~\ref{wedderburn}(a),(b). Moreover, $C_i^I\subseteq \mathbb{K}[N]\cap T_i=L_i=H_{\mathbb{K}[N]}(W_{i1}) \oplus \cdots \oplus  H_{\mathbb{K}[N]}(W_{i{t_i}})$. Again, by the structure of the semisimple algebra $\mathbb{K}[N]$, Theorem~\ref{wedderburn}(a),(b) implies that $C_i^I$ is sum of some of the homogenous components $H_{\mathbb{K}[N]}(W_{i1}),  \dots ,  H_{\mathbb{K}[N]}(W_{i{t_i}})$ of $\mathbb{K}[N]$. W.l.o.g. assume that $H_{\mathbb{K}[N]}(W_{i1})\subseteq C_i^I$.

 Now if we take $l \in \{1, \dots, s_i\} \setminus I$, then  $H_{\mathbb{K}[G]}(V_{il})H_{\mathbb{K}[N]}(W_{i1})=0$ by Theorem~\ref{wedderburn}(a),(b). In particular, if we consider  $V_{il}$ as $\mathbb{K}[N]$-module, then $V_{il}H_{\mathbb{K}[N]}(W_{i1})=0$, so that $H_{\mathbb{K}[N]}(W_{i1})\subseteq \op{Ann}_{\mathbb{K}[N]}(V_{il})=\{x\in \mathbb{K}[N]\mid V_{il}x=0\}$, the annihilator of $V_{il}$ as $\mathbb{K}[N]$-module. We observe that   $\op{Ann}_{\mathbb{K}[N]}(V_{il})=\bigcap_{m=1}^{t_i} \op{Ann}_{\mathbb{K}[N]}(W_{im})$, and so $$\op{Ann}_{\mathbb{K}[N]}(V_{il})=\op{Ann}_{\mathbb{K}[N]}(V_{ij})$$ for all $j \in \{1, \dots, s_i\}$. Thus, $T_iH_{\mathbb{K}[N]}(W_{i1}) =0$. In particular, since $H_{\mathbb{K}[N]}(W_{i1})\subseteq L_i\subseteq T_i$, we deduce that $H_{\mathbb{K}[N]}(W_{i1})=1_{T_i}H_{\mathbb{K}[N]}(W_{i1})=0$, which is a contradiction.
We obtain therefore that $C_i^I=0$. This implies in particular the claim, and so the contradiction $\ce{L_i}{T_i}=0$, which  proves that $\pi_r\circ \op{in}$ is onto.

We consider $\pi_1\circ \op{in}$, which is  onto with kernel
$$\ce{L_i}{T_i} \cap \op{Ker}(\pi_1) = \ce{L_i}{T_i} \cap \left(\bigoplus_{s=2}^{s_i} Z_{is}\right) \subseteq \mathbb{K}[N] \cap \left(\bigoplus_{s=2}^{s_i} H_{\mathbb{K}[G]}(V_{is})\right) = 0.$$
Hence, $\pi_1 \circ \op{in}$ is an isomorphism, and then $\op{dim}_{\mathbb{K}}(\ce{L_i}{T_i}) = 1$, as wanted.
\end{proof}

\bigskip

Certainly, Theorem~\ref{main_theorem} applies for the group algebra $\mathbb{C}[G]$ over the field $\mathbb{K}=\mathbb{C}$ of complex numbers.  Lemma \ref{lemma_classes} and Theorem \ref{main_theorem} provide in particular an  alternative proof for Corollary~\ref{corolario Brauer}, without using character theory.

\section{G-character tables of normal subgroups}\label{G-tables}

In the rest of the paper we focus on characters of groups, which are always considered over the field $\mathbb{C}$ of complex numbers. We adhere to~\cite{ISA} as main reference on this theory. We state versions for characters of Lemmas~\ref{conjugate}~and~\ref{Irr(G)}.

Let $N$ be a normal subgroup of a group $G$. Let $\theta\in \irr{N}$ and  $\{\theta^{g_i}\mid i=1,\dots, t\}$ be the orbit of $\theta$ under the action by conjugation of $G$ on $\irr{N}$, where $\{g_i\in G\mid i=1,\dots, t\}$ is a right transversal in $G$ of $I_G(\theta)$. We say that a character $\vartheta$ of $N$ is {\it minimal $G$-invariant} if it is $G$-invariant, i.e. $\vartheta^g=\vartheta$ for all $g\in G$, and there is no $G$-invariant constituent of $\vartheta$, different of $\vartheta$.

\begin{lemma}\label{conjugate-characters} With the notation above, the following assertions hold:
\begin{enumerate}\item The character $\widehat{\theta}:=\sum_{i=1}^t\theta^{g_i}$ is a minimal $G$-invariant character of $N$.
\item Every minimal $G$-invariant character of $N$  can be constructed from one orbit of irreducible characters in this way, and this defines a one-to-one correspondence between the set of orbits in the action by conjugation of $G$ on  $\irr{N}$ and the set of minimal $G$-invariant characters of $N$.
\item Every $G$-invariant character of $N$ is a sum of  minimal $G$-invariant characters of $N$.
\end{enumerate}
\end{lemma}

From now on, for any $\theta\in \irr{N}$ and with the previous notation, we write $\widehat{\theta}=\sum_{i=1}^t\theta^{g_i}$ for the minimal $G$-invariant character of $N$ corresponding to the orbit of $\theta$ in the action by conjugation of $G$ on  $\irr{N}$.

\begin{lemma}\label{Irr(G)-characters} Let $N$ be a normal subgroup of a group $G$. For $\chi,\phi\in \irr{G}$ the following statements are equivalent:
\begin{itemize}\item[(i)] $[\chi_N,\phi_N]\neq 0$.
\item[(ii)] There exist $\theta\in \irr{N}$ and positive integers $e_\chi$ and $e_\phi$ such that $\chi_N= e_\chi\, \widehat{\theta}$ and $\phi_N= e_\phi\, \widehat{\theta}$.
\item[(iii)] There is a rational number $c$  such that $\chi_N= c\, \phi_N$.
\end{itemize}
\end{lemma}

The previous lemma allows to introduce an equivalence relation on the set $\irr{G}$ as follows:

\begin{definition}\label{relEqIrr(G)}  Let $N$ be a normal subgroup of a group $G$. Two irreducible characters  $\chi,\phi$ of $G$ are defined to be {\it equivalent respect to $N$} if they satisfy any of the conditions in Lemma~\ref{Irr(G)-characters}.
\end{definition}

It is worth emphasizing the following facts: The equivalence class of each $\chi\in \irr{G}$ is associated to a  minimal $G$-invariant character  $\widehat{\theta}$ of $N$, with $\theta$ an irreducible constituent of $\chi_N$, and this defines  a one-to-one correspondence between the set of equivalence classes of irreducible characters of $G$ and the set of minimal $G$-invariant characters of $N$.
    Besides, the equivalence class of each $\chi\in \irr{G}$ is the set $\irr{G |\theta}=\{\phi\in \irr{G}\mid [\phi_N,\theta]\neq 0\}$.

In view of the previous two results, we introduce the following concept.

\begin{definition}\label{G-character table}
Let $N$ be a normal subgroup of a group $G$. A $G$\textbf{-character table of }$N$ is any (square) matrix $\textsf{X}=(x_{ij}) \in \textup{M}_k(\mathbb{C})$ with entries $$x_{ij}=\chi_i(n_j),\quad 1\leq i,j \leq k, $$ where  $\{n_1^G, \ldots, n_k^G\}$  is the set of $G$-conjugacy classes of $N$ and $\Delta = \{\chi_1=1_G, \chi_2, \dots, \chi_k\}$ denotes a representative system of the equivalence classes in the equivalence relation respect to $N$, defined on $\irr{G}$.

\end{definition}

 In the next results we will adhere to the following notation.

\begin{notation}\label{notation_tables}
Let $N$ be a normal subgroup of a group $G$. Denote:
\begin{itemize}\item  $\{n_1^G, \ldots, n_k^G\}$  the set of $G$-conjugacy classes of $N$;
 \item $\textsf{D}=(d_{ij})$ a diagonal matrix with entries $d_{ij}=\delta_{ij}|n_i^G|$, where $\delta_{ij}$ is the Kronecker delta function, for $i,j\in \{1, \ldots, k\}$;
\item $\Delta = \{\chi_1=1_G, \chi_2, \dots, \chi_k\}$ a representative system of the equivalence classes in the equivalence relation respect to $N$, defined on $\irr{G}$;
\item $\Omega= \{\theta_i \in \irr{N} \mid 1 \leq i \leq k\}$ such that $\chi_i \in \irr{G|\theta_i}$, for each $i \in \{1, \dots, k\}$;
\item $t_i=|G:I_{G}(\theta_i)|$, $e_{i} = [(\chi_i)_N, \theta_i]\neq 0$, $1\le i\le k$.
\item $\textsf{X}$ the $G$-character table of $N$ constructed from $\Delta$, and $\overline{\textsf{X}}^t$ its transposed conjugate matrix; i.e if $\textsf{X}=(x_{ij})$, then $\overline{\textsf{X}}^t=(y_{ij})$ being $y_{ij}=\overline{x_{ji}}$ the complex conjugate of $x_{ji}$, $1\le i,j\le k$.
\item  $\Lambda_{\textsf{X}}=\text{diag}(\lambda_1,\dots,\lambda_k)$ the  diagonal matrix with entries $\lambda_i=\abs{N}t_ie_i^2$, $ 1\le i\le k.$
\end{itemize}
\end{notation}

With this notation, we observe that the number of possible $G$-character tables (not necessarily different) of $N$ is $\prod_{i=1}^{k}|\irr{G |\theta_i}|,$ where  $\theta_i\in \irr{N}$ such that $[{\chi_i}_N,\theta _i]\neq 0$, for each $i=1,\dots, k$.
\smallskip

If we consider the submatrix $\textsf{Y}$ of the character table of $G$ defined by the columns corresponding to the $G$-conjugacy classes of the normal subgroup $N$, then  by Lemma~\ref{Irr(G)-characters}  one just needs to look at the proportional rows by rational numbers of $\textsf{Y}$  to identify the equivalence classes in the equivalence relation respect to $N$, defined on $\irr{G}$, and to know their sizes as well as the relations between the ramification numbers corresponding to equivalent irreducible characters of $G$ respect to $N$. Indeed, in Theorem~\ref{square}, it will be proven  that $G$-character tables are non-singular matrices. Therefore, proportional rows of  $\textsf{Y}$ identify each equivalence class on $\irr{G}$, and this proportionality is only possible by rational numbers. We gather those facts in the next corollary.

\begin{corollary}\label{sizes} Let $N$ be a normal subgroup of a group $G$. The equivalence classes in the equivalence relation respect to $N$, defined on $\irr{G}$, are detected in the character table of the group $G$. In particular, with Notation~\ref{notation_tables}, the size of $\irr{G|\theta_i}$, for every $\theta_i\in \Omega$, is determined by the character table of $G$, as well as the rational numbers $\frac{e_\chi}{e_\phi}$ whenever $\chi,\phi\in \irr{G|\theta_i}$,  with $\chi_N= e_\chi\, \widehat{\theta_i}$ and $\phi_N= e_\phi\, \widehat{\theta_i}$.
\end{corollary}

\begin{remark}\label{zeros} We  emphasize the information that $G$-character tables of normal subgroups can contribute in the study of the normal structure of groups and the influence of $G$-conjugacy classes. For instance, from Theorem~A of \cite{FGS} one easily deduces that if a $G$-character table (and therefore all of them) of a normal subgroup $N$ of the group $G$ has no zeros, then $N$ is nilpotent.
\end{remark}

\begin{theorem}\label{square}
With  Notation~\ref{notation_tables}, it holds that  \textup{$\Lambda_{\textsf{X}}=\textsf{X}\textsf{D}\overline{\textsf{X}}^t.$} In particular, \textup{$\textsf{X}$} is non-singular.
\end{theorem}

\begin{proof}
Let $\chi_i, \chi_j \in \Delta$. Then there exist elements $g_1=1, g_2, \ldots, g_{t_i}, h_1=1, h_2, \ldots, h_{t_j}\in G$ such that
$$\chi_i(x)=e_{i}(\theta_i^{g_1}(x)+ \theta_i^{g_2}(x)+\cdots+\theta_i^{g_{t_i}}(x)), $$
and
$$\chi_j(x)=e_{j}(\theta_j^{h_1}(x)+ \theta_j^{h_2}(x)+\cdots+\theta_j^{h_{t_j}}(x)),$$ for all $x\in N$.
If $i \neq j$, then by the orthogonality relations we obtain
\begin{eqnarray*}
\sum_{x \in N}\chi_i(x)\overline{\chi_j(x)}&=&\sum_{x \in N} e_{i} e_{j} \sum_{r=1}^{t_i} \sum_{s=1}^{t_j} \theta_{i}^{g_r}(x)\overline{\theta_j^{h_s}(x)}  \\
&=& e_{i} e_{j}\sum_{r=1}^{t_i} \sum_{s=1}^{t_j} \left( \sum_{x \in N} \theta_{i}^{g_r}(x)\overline{\theta_j^{h_s}(x)}\right)   \\
& = & e_{i} e_{j}\sum_{r=1}^{t_i} \sum_{s=1}^{t_j}  \abs{N}[\theta_{i}^{g_r},\theta_j^{h_s}]=0.
\end{eqnarray*}
If $i=j$, then we get
\begin{eqnarray*}
\sum_{x \in N}\chi_i(x)\overline{\chi_i(x)}&=&\sum_{x \in N} e_{i}^2 \sum_{r=1}^{t_i} \sum_{s=1}^{t_i} \theta_{i}^{g_r}(x)\overline{\theta_i^{g_s}(x)}=e_{i}^2\sum_{r=1}^{t_i} \sum_{s=1}^{t_i} \left( \sum_{x \in N} \theta_{i}^{g_r}(x)\overline{\theta_i^{g_s}(x)}\right) = \\
& = & e_{i}^2\sum_{r=1}^{t_i} \left( \sum_{x \in N} \theta_{i}^{g_r}(x)\overline{\theta_i^{g_r}(x)}\right) = e_{i}^2t_i\abs{N}.
\end{eqnarray*}
Thus, if $\chi_i, \chi_j \in \Delta$, it follows $e_{i} ^2 t_i |N| \delta_{ij}=\sum_{x \in N} \chi_i(x)\overline{\chi_j(x)}=\sum_{s=1}^k|n_s^G|\chi_i(n_s)\overline{\chi_j(n_s)}$,
and the first assertion follows. In particular, the determinant of $\textsf{X}\textsf{D}\overline{\textsf{X}}^t$ is non-zero, and therefore $\textsf{X}$ is non-singular.
\end{proof}
\smallskip

As a first consequence we notice  that a  $G$-character table of $N$ allows to know whether an irreducible character of $G$ restricts into an irreducible character of $N$.

\begin{corollary}\label{restriction} With  Notation~\ref{notation_tables}, let $\chi\in \irr{G}$ and $i\in \{1,\dots,k\}$ such that $\chi=\chi_i\in \Delta$. Then $\chi_N\in \irr{N}$ if and only if  $\lambda _i=|N|$.
\end{corollary}

From  Theorem~\ref{square} we deduce that some arithmetical relations between the aforementioned integers $t_i$, $e_{i}$ and $\theta_i(1)$ can be read off the matrix $\Lambda_{\textsf{X}}$, for the  $G$-character table $\textsf{X}$ of $N$.

\begin{corollary}\label{relations}
 With Notation~\ref{notation_tables}, the next integer relations hold, where the corresponding right sides can be computed from the  $G$-character table $\textsf{X}$ of $N$:
\begin{align}
e_{i}^2t_i&=\dfrac{\lambda_i}{\abs{N}} \tag{$A_i$} \\
t_i\theta_i(1)^2&=\dfrac{\abs{N}\chi_i(1)^2}{\lambda_i} \tag{$B_i$} \\
\dfrac{\theta_i(1)}{e_{i}}&=\dfrac{\abs{N}\chi_i(1)}{\lambda_i} \tag{$C_i$}
\end{align}
$1\le i\le k$.
\end{corollary}

\begin{proof} The relation $(A_i)$ follows directly from the definition of $\lambda_i$, $1\le i\le k$, as well as $(B_i)$ and $(C_i)$, using in addition the fact that $\chi_i(1)=e_{i}t_i\theta_i(1)$, for every $i$. The corresponding right sides of these equations can be computed from  the  $G$-character table $\textsf{X}$ of $N$, since 
$\Lambda_{\textsf{X}}=\textsf{X}\textsf{D}\overline{\textsf{X}}^t$ by Theorem~\ref{square}.
\end{proof}

Let $\widehat{\theta}$ be a minimal $G$-invariant character of $N$, $\theta\in \irr{N}$. Then, with Notation~\ref{notation_tables}, $\widehat{\theta}=\widehat{\theta}_i$ for some $\theta_i\in \irr{N}$, $i\in \{1,\dots,k\}$. We notice that $\widehat{\theta}_i(1)=t_i\theta_i(1)$, and so the prime divisors of $\widehat{\theta}(1)$ are known from equation  $(B_i)$ in Corollary~\ref{relations}. Hence, the following result follows, together with It\^o-Michler's theorem,  which asserts that, for a group $X$, if a prime $p$ does not divide $\chi(1)$ for all $\chi\in \irr{X}$,  then $X$ has an abelian normal  Sylow $p$-subgroup (\cite[Theorem 5.4]{I-M}).

\begin{corollary}\label{primesinvariants} Let $N$ be a normal subgroup of a group $G$. Then the prime divisors of $\widehat{\theta}(1)$ for every minimal $G$-invariant character  $\widehat{\theta}$ of $N$   are known from the character table of $G$. Moreover, if a prime $p$ does not divide $\widehat{\theta}(1)$ for every minimal $G$-invariant character $\widehat{\theta}$ of $N$, then $N$ has an abelian normal  Sylow $p$-subgroup.
\end{corollary}

The relations in Corollary~\ref{relations} yield also significant information of $N$ from its $G$-character tables in certain situations, as we show below.

\begin{remark}\label{remarks}
\begin{enumerate}
\item If $N$ is a Hall subgroup of $G$, then $t_i$ and $\theta_i(1)$  are coprime numbers, for each $i=1,\dots,k$, since $t_i$ divides $|G:N|$ and $\theta_i(1)$ divides $|N|$. Consequently, the integers $t_i$ and $\theta_i(1)$ can be computed from  the equations $(B_i)$, and then also the integers $e_i$ are known from $(A_i)$.

\item If $(A_i)$ is square-free, for some $i\in \{1,\dots,k\}$, then it follows that $e_{i}=1$ and $t_i=\frac{\lambda_i}{\abs{N}}$, so $\theta_i(1)=\frac{\abs{N}\chi_i(1)^2}{\lambda_i}$. In particular, the integers $t_i$, $e_{i}$ and $\theta_i(1)$ can be computed from a $G$-character table of $N$. Analogously, if $(B_i)$ is square-free,  for some $i\in \{1,\dots,k\}$, then the integers $t_i$, $e_{i}$ and $\theta_i(1)$ can be also computed.
\end{enumerate}
\end{remark}

In brief, with Notation~\ref{notation_tables}, the numbers $\lambda_i$, $1\le i\le k$, as well as $|N|$, are known from the character table of the group $G$, and so are also the relations between the important parameters $e_i, t_i, \theta_i(1)$, $1\le i\le k$, given in Corollary~\ref{relations}, as mentioned above. As a consequence, in particular situations when one of these parameters, for some $i\in \{1,\dots,k\}$, either $e_i$, $t_i$ or $\theta_i(1)$, is known, then all three of them  are known. But one can not expect that either the $G$-character tables of $N$ or even the character table of $G$ provides all these values with no additional information. This would mean, for instance, a positive answer to Problem 10 of Brauer's famous list in \cite{B}, which asked whether or not $N$ is abelian can be decided from the character table of the group. It is  known that this is not the case (c.f. \cite{S,DI}). In particular, the authors in \cite{DI} consider the concept of {\it automorphism of the character table} of a group, which is a permutation of rows and columns of the character table, considered as a matrix, such that the resulting matrix coincides with the initial one. Two normal subgroups of the group are then called {\it character-table-isomorphic} if there is an automorphism of the character table such that the image of one of the subgroups (as collection of conjugacy classes) is the other one. In that reference examples of groups are given, containing corresponding  two character-table-isomorphic normal subgroups, such that only one of the subgroups is abelian. Obviously character-table-isomorphic normal subgroups of a group $G$ have equal $G$-character tables, with equal corresponding diagonal matrices $\textsf{D}$ of $G$-conjugacy class sizes. (See Example~\ref{example2}.)

\section{ $G$-invariant table and minimal $G$-invariant characters of a normal subgroup}\label{sourceG-tables}

According with Lemma~\ref{conjugate-characters} and Notation~\ref{notation_tables}, we denote $$\Min{N}=\{\widehat{\theta}\mid \theta \in \irr{N}\}=\{\widehat{\theta}_i\mid i=1,\dots,k\}$$ the set of minimal $G$-invariant characters of $N$, which has cardinality $k$. It is clear from the definition that a $G$-character table of the normal subgroup $N$ is not unique, unless $G=N$. The description of each $G$-character table  depends upon the choice of a representative system of the equivalence classes in the equivalence relation respect to $N$ on $\irr{G}$. But it is also obvious that two $G$-character tables have proportional corresponding rows, by rational numbers given by the ramification numbers. One can then think of a {\it  $G$-invariant table of $N$} as defined next.

\begin{definition}\label{sourcetable} With Notation~\ref{notation_tables}, the \textbf{ $G$-invariant table of $N$} is defined to be the matrix $\widehat{\textsf{X}}= (y_{ij}) \in \textup{M}_k(\mathbb{C}),$
with entries $$y_{ij}=\widehat{\theta}_i(n_j),\quad 1\leq i,j \leq k. $$
\end{definition}

Observe that this $G$-invariant table of $N$ is unique and, comparing with the  $G$-character table $\textsf{X}$ of $N$, each row of $\textsf{X}$ is a multiple of the corresponding one of $\widehat{\textsf{X}}$. In particular, $\widehat{\textsf{X}}$ is non-singular. We notice that the knowledge of the  $G$-invariant table of $N$ from some $G$-character table of $N$ is equivalent to the knowledge of the ramifications numbers.

When $G=N$, the corresponding $G$-invariant table  is the character table of $G$, and so the  $G$-invariant table might be seen as a version of the character table of the group for normal subgroups, where the minimal $G$-invariant characters of the normal subgroup play the role of the irreducible characters of $G$. Indeed, for a normal subgroup $N$ of a group $G$, we can consider
\begin{align*}\text{cf}_G(N)&=\{f\mid f\text{ class function of }N,\ f^g=f \text{ for all } g\in G\}\\&=\{ f:N\longrightarrow \mathbb{C}\mid f(n^g)=f(n) \text{ for all }n\in N \text{ and all }g\in G\},\end{align*}
the set of $G$-invariant class functions of $N$, which are exactly  the class functions of $N$ constant on $G$-conjugacy classes.
It is straightforward to check that $\text{cf}_G(N)$ is a subspace  of the $\mathbb{C}$-vector space  of class functions of $N$, with basis the set of irreducible characters of $N$ (c.f. \cite[Theorem 2.8]{ISA}), and $\text{cf}_G(N)$ contains  the set $\Min{N}$ of minimal $G$-invariant characters of $N$, which are proven to form a basis as stated below.

\begin{proposition}\label{G-classfunctions} Let $N$ be a normal subgroup of a group $G$. If $f\in \textup{cf}_G(N)$ then $f$ can be uniquely expressed in the form $$f=\sum_{\widehat{\theta}\in \Min{N}}a_{\widehat{\theta}}\,\widehat{\theta},$$
where $a_{\widehat{\theta}}\in \mathbb{C}$. Furthermore, $f$ is a $G$-invariant character of $N$ if and only if all of the $a_{\widehat{\theta}}$ are nonnegative integers and $f\neq 0$.
\end{proposition}

\begin{proof} On the one hand it is easily derived that $\text{cf}_G(N)$ forms a $\mathbb C$-vector space whose dimension is the number $k$ of $G$-conjugacy classes in $N$. Moreover, by Theorem~\ref{square} it follows that the set $\Min{N}$ forms an independent system of cardinality $k$ contained in $\text{cf}_G(N)$, which implies that it forms a basis of $\text{cf}_G(N)$, and proves the first part of the result. The rest follows from Lemma~\ref{conjugate-characters}(3).
\end{proof}

We close this section with the next generalisation of Brauer's theorem \ref{brauer_theorem} and further applications. More precisely, the  $G$-invariant table of a normal subgroup may play the role of the character table of the group $G$ in the proof of Brauer's theorem \ref{brauer_theorem}, as given in \cite[Theorem (6.32)]{ISA}. With this idea and analogous arguments the next generalization can be stated. We include the proof for the sake of completeness.

\begin{theorem}\label{Brauer}
Let $N$ be a normal subgroup of a group $G$. Let
$A$ be a group which  acts on the set $\Min{N}$ of minimal $G$-invariant characters of $N$ and on the set of $G$-conjugacy classes of $N$.  Assume that $\widehat{\theta}\,^a(n^a)=\widehat{\theta}(n)$ for all $\widehat{\theta}\in\Min{N}$, $a\in A$, and $n\in N$; where $n^a$ is an element of $(n^G)^a$. Then for each $a\in A$, the number of fixed characters of $\Min{N}$ is equal to the number of fixed $G$-conjugacy classes.
\end{theorem}

\begin{proof} With Notation~\ref{notation_tables}, we consider the  $G$-invariant table $\widehat{\textsf{X}}$ as defined above. For each $a\in A$, let $P_a=(p_{ij})$ where $p_{ij}=0$ unless $\widehat{\theta}_i\,^a=\widehat{\theta}_j$, in which case $p_{ij}=1$. Similarly, we define $Q_a=(q_{ij})$ where $q_{ij}=1$ if $(n_i^G)^a=n_j^G$, and zero otherwise. We write $n_i^a=n_j$ if $(n_i^G)^a=n_j^G$. It is not difficult to see that the $(u,v)$-entry of the matrix $P_a \widehat{\textsf{X}}$ is $$\displaystyle\sum_{i=1} ^k p_{ui}\widehat{\theta}_i(n_v)=\widehat{\theta}_u\,^a(n_v).$$ Similarly, the $(u,v)$-entry of $\widehat{\textsf{X}}Q_a$ is $$\displaystyle\sum_{j=1}^k\widehat{\theta}_u(n_j)q_{jv}=\widehat{\theta}_u(n_v^{a^{-1}}).$$ Our hypotheses lead to $P_a\widehat{\textsf{X}}=\widehat{\textsf{X}}Q_a$, and since $\widehat{\textsf{X}}$ is non-singular, then $Q_a=\widehat{\textsf{X}}^{-1}P_a\widehat{\textsf{X}}$. We conclude that the traces of $Q_a$ and $P_a$ are equal, and these are precisely the number of fixed points of the action of $a$ on $\Min{N}$ and the set of $G$-conjugacy classes of $N$.
\end{proof}

In the same spirit, we are also able to obtain a version for $G$-conjugacy classes of the well-known result which states that the number of real classes of a group $G$ is equal to the number of real valued irreducible characters of $G$. We recall that an element $g\in G$ is said to be {\it real} if $g$ is conjugate in $G$ to its inverse $g^{-1}$; in this case, the corresponding conjugacy class $g^G$ is called {\it real}. Also a character $\vartheta$ of a group $X$ is {\it real valued} if $\vartheta(x)$ is a real number for all $x\in X$.

We point out that if $N$ is a normal subgroup of a group $G$, the number of real valued minimal $G$-invariant characters of $N$ coincides with the number of  irreducible characters of $G$, which are not equivalent respect to $N$, and whose restrictions to $N$ are real valued, as well as with the number of real valued rows in any $G$-character table of $N$, by Lemmas~\ref{conjugate-characters},~\ref{Irr(G)-characters} and Definition~\ref{G-character table}.

\begin{corollary}\label{CorBrauer}
Let $N$ be a normal subgroup of a group $G$. Then the number of real $G$-conjugacy classes of $N$ is equal to the  number of real valued minimal $G$-invariant characters of $N$.
\end{corollary}

\begin{proof}
Let $A=\langle \sigma \rangle$ be a cyclic group of order $2$ that acts on the set $\Min{N}$ of all minimal $G$-invariant characters of $N$, and on the set of $G$-conjugacy classes of $N$, such that  $\widehat{\theta}\,^{\sigma}(n)=\widehat{\theta}(n^{-1})$, and $(n^G)^{\sigma}=(n^{-1})^G$, for all $n\in N$. Note that $\sigma$ fixes a $G$-conjugacy class $n^G$ if and only if $n^G$ is real. Also the characters $\widehat{\theta}\in \Min{N}$ which are fixed by $\sigma$ are exactly the  real valued characters in $\Min{N}$. It follows from Theorem~\ref{Brauer} that the number of real $G$-conjugacy classes of $N$ is equal to the number of real valued characters of $\Min{N}$.
\end{proof}

The question about the relation between the number of real conjugacy classes and the structure of the group has been previously considered. Groups with all elements real are called {\it ambivalent} and have been studied  by Berggren and others (e.g., \cite{Ber}). In the opposite situation, Burnside noticed that a group $G$ is of odd order if and only if $G$ has a unique real conjugacy class (which is of course the one formed by the identity of the group). Iwasaki \cite{Iwa} characterised those groups with exactly two real conjugacy classes, in terms of the normality of a Sylow $2$-subgroup.
We wonder about the influence of real $G$-conjugacy classes on a normal subgroup $N$ of a group $G$. If $x\in N$ and $x$ is real in $N$ then $x$ is obviously real in $G$, but the converse is easily checked not to be true in general, as the next example shows.

\begin{example}
\label{example_real}
Let $G=\langle a\rangle  \rtimes \langle b \rangle$ be the holomorph of $C_5\cong N=\langle a\rangle$ the cyclic group of order $5$, with  $a^b=a^2$. It is clear the $N$ has no real conjugacy class different from the trivial one. However there are two $G$-conjugacy classes, namely $1^G$ and $a^G$, and they are real. Indeed,  one can check that a $G$-character table of $N$ is given  by the matrix $\begin{bmatrix}1&1\\4&-1\end{bmatrix}$, while the character table of $N$ contains non-real values.
\end{example}

The next results show that real $G$-conjugacy classes of normal subgroups can provide  structural information about them.

\begin{lemma}
\label{odd_order}
Let $N$ be a normal subgroup of a group $G$. If there is a unique real $G$-conjugacy class in $N$, then $N$ is of odd order, but the converse is not true in general.
\end{lemma}

\begin{proof} If $N$ had even order, then $N$ would contain an involution, which would be a real element different from the identity.
Example~\ref{example_real} shows that the converse is not true in general.
\end{proof}

\begin{corollary}
\label{teo_iwa}
Let $N$ be a normal subgroup of a group $G$. Assume that there are exactly two real $G$-conjugacy classes in $N$. Then $N$ has a normal Sylow $2$-subgroup.
\end{corollary}

\begin{proof}
We argue by induction on $\abs{G}$. We may assume that $N$ is of even order and consider an involution $x\in N$. Then $x^G$ is a non-trivial real $G$-conjugacy class and, by hypothesis,   $x^G$ is the set of  all involutions of $N$.
Set $M:=1\cup x^G\subseteq N$. We claim that $M$ is a (elementary abelian normal) subgroup of $G$, and it is enough to prove that $ab\in M$, if  $a,b\in M$ are involutions. Since  $(ab)^a=ba=b^{-1}a^{-1}=(ab)^{-1}$, it follows that $ab$ is real, which implies by hypothesis that  $ab\in M$.

Let us denote by ``bar'' the images in the factor group $\overline{G}:=G/M$. If $|\overline{N}|$ is odd, then $M\in\operatorname{Syl}_2(N)$ and we are done. Hence we may assume that $\overline{N}$ has a non-trivial $\overline{G}$-conjugacy class that is real. We claim that the number of real valued rows in a $\overline{G}$-character table of $\overline{N}$ is not larger than that in a $G$-character table of $N$. If $\overline{\chi}$ is an irreducible character of $\overline{G}$, then $\overline{\chi}(gM)=\chi(g)$ for all $g\in G$, where $\chi$ is an irreducible character of $G$ containing $M$ in its kernel. It is trivial that $\overline{\chi}_{N/M}$ is real valued if and only if $\chi_N$ is real valued. Moreover, if $\overline{\chi}_1, \overline{\chi}_2\in \irr{\overline{G}}$, with corresponding characters $\chi_1,\chi_2\in \irr{G}$ containing $M$ in their kernels, then $\overline{\chi}_1$ and $\overline{\chi}_2$ are equivalent respect to $N/M$ if and only if $\chi_1$ and $\chi_2$ are equivalent respect to $N$. The claim is now clear. (Indeed, if $\chi\in \irr{G}$ contains $M$ in its kernel, and $\phi\in \irr{G}$ is equivalent to $\chi$ respect to $N$, then $M$ is contained in the kernel of $\phi$.) Corollary~\ref{CorBrauer} implies that $\overline{N}$ has exactly two real $\overline{G}$-conjugacy classes, and the result follows by induction.
\end{proof}

It is worth to highlight that a normal subgroup $N$ of a group $G$ may satisfy the hypotheses of Corollary~\ref{teo_iwa} and, however, it might not satisfy those of Iwasaki's theorem, i.e. it might not have exactly two real conjugacy classes, as  Example~\ref{example_real} easily shows.

\section{Some illustrative examples}\label{Section examples}

The next examples show the application of our results to particular groups; especially the relations in Corollary~\ref{relations}, which may be very useful to obtain information about normal subgroups. The first example shows that it is possible to completely describe the character table of the considered normal subgroup, although in general it is not determined by the character table of  the group, as well as to derive the  existence of an abelian normal Sylow subgroup in the normal subgroup. In this example, the actions by conjugation of the group,  on the set of irreducible characters and on the set of conjugacy classes of the normal subgroup,  happen to be the trivial actions. In the second example we apply our computations to one of the groups, and its corresponding normal subgroups, previously mentioned, following Remark~\ref{remarks}, which appears in \cite[Section 3.2]{DI}. As expected, it shows up that the relations in  Corollary~\ref{relations} do not completely determine the parameters $e_i,t_i$ and $\theta_i(1)$, $1\le i\le k$.
Observe that we show techniques to obtain information from  the character table of the group, though certainly the knowledge of the group, as well as other techniques, might provide other possibilities.

Unless for the labels for irreducible characters of $G$, we follow the notation introduced previously in the paper;  in particular, Notation \ref{notation_tables}.

\begin{example}\label{example1}
Let $A$ be a dihedral group of order 8, and let $B$ be an alternating group on $4$ letters. Consider $G=A\times B$. The character table of $G$ is the $20\times 20$ matrix appearing in Table~\ref{tabla11}, where $\zeta$ denotes a primitive $3$-root of unity, and the second row corresponds to the sizes of the conjugacy classes. Let $N$ be the intersection of the kernels of $\chi_1, \chi_2,\chi_3$ and $\chi_4$, so $N$ is a normal subgroup of $G$ which is the union of the conjugacy classes $1A, 3A, 2C, 2D, 3B, 6C, 2G, 6F$. It follows that $\abs{N}=24$.

\setlength{\arraycolsep}{1.5pt}

\begin{table}
\begin{footnotesize}
$$\begin{array}{ccccccccccccccccccccc}
& 1A & 2A & 2B & 3A & 2C & 2D & 4A &  6A & 2E &  6B & 2F & 3B &  6C & 2G & 12A & 4B  & 6D  & 6E &  6F & 12B \\

& 1& 2& 2& 4& 1& 3& 2& 8& 6& 8& 6& 4& 4& 3& 8& 6& 8& 8& 4& 8 \\

&&&&&&&&&&&&&&&&&&&&\\

\chi_1 & 1 & 1 & 1 & 1 & 1 & 1 & 1 & 1 & 1 & 1 & 1 & 1 & 1 & 1 & 1 & 1 & 1 & 1 & 1 & 1 \\

\chi_2 &1 & -1 & -1 & 1 & 1 & 1 & 1 & -1 & -1 & -1 & -1 & 1 & 1 & 1 & 1 & 1 & -1 & -1 & 1 & 1 \\

\chi_3 &1 & -1 & 1 & 1 & 1 & 1 & -1 & -1 & -1 & 1 & 1 & 1 & 1 & 1 & -1 & -1 & -1 & 1 & 1 & -1 \\

\chi_4 &1 & 1 & -1 & 1 & 1 & 1 & -1 & 1 & 1 & -1 & -1 & 1 & 1 & 1 & -1 & -1 & 1 & -1 & 1 & -1 \\

\chi_5 &1 & -1 & -1 & \zeta^2 & 1 & 1 & 1 & -\zeta^2 & -1 & -\zeta^2 & -1 & \zeta & \zeta^2 & 1 & \zeta^2 & 1 & -\zeta & -\zeta & \zeta & \zeta \\

\chi_6 &1 & -1 & -1 & \zeta & 1 & 1 & 1 & -\zeta & -1 & -\zeta & -1 &\zeta^2& \zeta & 1 & \zeta & 1 & -\zeta^2& -\zeta^2&\zeta^2&\zeta^2\\

\chi_7 &1 & -1 & 1 &\zeta^2& 1 & 1 & -1 & -\zeta^2& -1 &\zeta^2& 1 & \zeta &\zeta^2& 1 & -\zeta^2& -1 & -\zeta & \zeta & \zeta & -\zeta \\

\chi_8 &1 & -1 & 1 & \zeta & 1 & 1 & -1 & -\zeta & -1 & \zeta & 1 &\zeta^2& \zeta & 1 & -\zeta & -1 & -\zeta^2&\zeta^2&\zeta^2& -\zeta^2\\

\chi_9 &1 & 1 & -1 &\zeta^2& 1 & 1 & -1 &\zeta^2& 1 & -\zeta^2& -1 & \zeta &\zeta^2& 1 & -\zeta^2& -1 & \zeta & -\zeta & \zeta & -\zeta \\

\chi_{10} &1 & 1 & -1 & \zeta & 1 & 1 & -1 & \zeta & 1 & -\zeta & -1 &\zeta^2& \zeta & 1 & -\zeta & -1 &\zeta^2& -\zeta^2&\zeta^2& -\zeta^2\\

\chi_{11} &1 & 1 & 1 &\zeta^2& 1 & 1 & 1 &\zeta^2& 1 &\zeta^2& 1 & \zeta &\zeta^2& 1 &\zeta^2& 1 & \zeta & \zeta & \zeta & \zeta \\

\chi_{12} &1 & 1 & 1 & \zeta & 1 & 1 & 1 & \zeta & 1 & \zeta & 1 &\zeta^2& \zeta & 1 & \zeta & 1 &\zeta^2&\zeta^2&\zeta^2&\zeta^2\\

\chi_{13} &2 & 0 & 0 & 2 & -2 & 2 & 0 & 0 & 0 & 0 & 0 & 2 & -2 & -2 & 0 & 0 & 0 & 0 & -2 & 0 \\

\chi_{14} &2 & 0 & 0 & 2\zeta^2 & -2 & 2 & 0 & 0 & 0 & 0 & 0 & 2\zeta & -2\zeta^2 & -2 & 0 & 0 & 0 & 0 & -2\zeta & 0 \\

\chi_{15} &2 & 0 & 0 & 2\zeta & -2 & 2 & 0 & 0 & 0 & 0 & 0 & 2\zeta^2 & -2\zeta & -2 & 0 & 0 & 0 & 0 & -2\zeta^2 & 0 \\

\chi_{16} &3 & -3 & -3 & 0 & 3 & -1 & 3 & 0 & 1 & 0 & 1 & 0 & 0 & -1 & 0 & -1 & 0 & 0 & 0 & 0 \\

\chi_{17} &3 & -3 & 3 & 0 & 3 & -1 & -3 & 0 & 1 & 0 & -1 & 0 & 0 & -1 & 0 & 1 & 0 & 0 & 0 & 0 \\

\chi_{18} &3 & 3 & -3 & 0 & 3 & -1 & -3 & 0 & -1 & 0 & 1 & 0 & 0 & -1 & 0 & 1 & 0 & 0 & 0 & 0 \\

\chi_{19} &3 & 3 & 3 & 0 & 3 & -1 & 3 & 0 & -1 & 0 & -1 & 0 & 0 & -1 & 0 & -1 & 0 & 0 & 0 & 0 \\

\chi_{20} &6 & 0 & 0 & 0 & -6 & -2 & 0 & 0 & 0 & 0 & 0 & 0 & 0 & 2 & 0 & 0 & 0 & 0 & 0 & 0
\end{array}$$
\end{footnotesize}
\caption{Character table of the group $G$ in Example~\ref{example1}.}
\label{tabla11}
\end{table}

Then there are  $8$ equivalence classes in the equivalence relation respect to $N$, defined on $\irr{G}$, and so also $8$ minimal $G$-invariant characters of $N$, denoted $\widehat{\theta_i}$, $i=1,\dots,8$, with $\Omega=\{\theta_i\mid i=1,\dots,8\}$ a representative system of the orbits in the action by conjugation of $G$ on $\irr{N}$. From the observation of proportional rows in the character table, when restricting to the described conjugacy classes, we obtain:
$$
\begin{array}{ll}
\irr{G|\theta_1}=\{\chi_1, \chi_2, \chi_3, \chi_4  \}, \quad & \irr{G|\theta_2}=\{\chi_5, \chi_7, \chi_9, \chi_{11}  \}, \\

\irr{G|\theta_3}=\{\chi_6, \chi_8, \chi_{10}, \chi_{12}  \}, \quad & \irr{G|\theta_4}=\{\chi_{13}  \}, \\

 \irr{G|\theta_5}=\{\chi_{14} \}, \quad & \irr{G|\theta_6}=\{\chi_{15}  \}, \\

 \irr{G|\theta_7}=\{\chi_{16}, \chi_{17}, \chi_{18}, \chi_{19},  \}, \quad & \irr{G|\theta_8}=\{\chi_{20}  \},
\end{array}$$

In particular, for each $i=1,\dots, 8$, we know the size of $ \irr{G|\theta_i}$. Also, whenever $\chi,\vartheta\in \irr{G|\theta_i}$, it holds that $$\chi_N=\vartheta_N=e_i\widehat{\theta_i}.$$

We may choose
$\{\chi_1, \chi_5, \chi_6,\chi_{13}, \chi_{14}, \chi_{15}, \chi_{16}, \chi_{20}\}$ a representative
system of the equivalence classes on $\irr{G}$, and construct the corresponding $G$-character table $\textsf{X}$ of $N$, which is the matrix of dimension $8\times 8$ in Table~\ref{tabla12}.

\begin{table}
\begin{footnotesize}
$$\begin{array}{ccccccccc}
& 1A& 3A& 2C& 2D& 3B& 6C& 2G& 6F \\
& 1&4&1&3&4&4&3&4 \\
&&&&&&&&\\

\chi_1 & 1 & 1 & 1 & 1 & 1 & 1 & 1 & 1 \\

\chi_5 & 1 & \zeta^2 & 1 & 1 & \zeta & \zeta^2 & 1 & \zeta \\

\chi_6 & 1 & \zeta & 1 & 1 & \zeta^2 & \zeta & 1 & \zeta^2 \\

\chi_{13} & 2 & 2 & -2 & 2 & 2 & -2 & -2 & -2 \\

\chi_{14} & 2 & 2\zeta^2 & -2 & 2 & 2\zeta & -2\zeta^2 & -2 & -2\zeta \\

\chi_{15} & 2 & 2\zeta & -2 & 2 & 2\zeta^2 & -2\zeta & -2 & -2\zeta^2 \\

\chi_{16} & 3 & 0 & 3 & -1 & 0 & 0 & -1 & 0 \\

\chi_{20} & 6 & 0 & -6 & -2 & 0 & 0 & 2 & 0
\end{array}$$
\end{footnotesize}
\caption{$G$-character table $\textsf{X}$ of $N$ in Example~\ref{example1}.}
\label{tabla12}
\end{table}

 Note also that $\textsf{D}=\text{diag}(1,4,1,3,4,4,3,4)$ is the associated diagonal matrix, whose diagonal entries are the sizes of the $G$-conjugacy classes that built $N$.

 According to Corollary~\ref{relations}, we obtain that the numbers appearing in the right side of relations $(A_i)$, $1\le i\le 8$, are $1,1,1,4,4,4,1,4$, respectively. A first conclusion is that for each $i\in \{1,2,3,7\}$, it holds that $$e_i=t_i=1,\text{ and so }\chi_N=\widehat{\theta_i}=\theta_i\in \irr{N},$$ whenever $\chi\in \irr{G|\theta_i}$. We compute now the numbers appearing in the right side of relations $(B_i)$, $1\le i\le 8$, and obtain $ 1,1,1,1,1,1,9,9$, respectively. As immediate consequence, we know that for every $i\in \{1,2,3,4,5,6\}$, it holds that $$t_i=1, \theta_i(1)=1, \text{ and so also } \widehat{\theta_i}=\theta_i.$$
 In particular, $\widehat{\theta_i}(1)=1$, for $i=1,\dots, 6$, and the only prime divisor of $\widehat{\theta_i}(1)$, for $i=7,8$ is $3$.

Besides, by Corollary~\ref{primesinvariants}, since $2$ does not divide $\widehat{\theta_i}(1)$ for all $i=1,\dots,8$, we can affirm that $N$ has an abelian normal Sylow $2$-subgroup. Note that, however, $2$ divides the degrees of $\chi_{13},\chi_{14}, \chi_{15}$ and $\chi_{20}$, all of them appearing in the $G$-character table considered.

 From equations $(A_i)$, $i=4,5,6$, and $(B_7)$, it follows that $e_i=2$, for $i=4,5,6$, and $\theta_7(1)=3$, respectively. Moreover, from $(A_8)$ and $(B_8)$, it follows that $t_8=1$, $e_8=2$ and $\theta_8(1)=3$. Also, $\widehat{\theta_8}=\theta_8$.

 Consequently, $\irr{N}=\Min{N}=\{\theta_i=\widehat{\theta_i}\mid i=1,\dots, 8\}$, with $\theta_i(1)=1$, $i=1,\dots 6$, $\theta_i(1)=3$, $i=7,8$, and  the character table of $N$ can be completely described in this particular example, having $n^N=n^G$ for all $n\in N$, and the irreducible characters:
\begin{align*}\theta_1&=(\chi_1)_N,&\theta_2&=(\chi_5)_N,&\theta_3&=(\chi_6)_N,&\theta_4&=\frac{1}{2}(\chi_{13})_N,\\ \theta_5&=\frac{1}{2}(\chi_{14})_N,
&\theta_6&=\frac{1}{2}(\chi_{15})_N,&\theta_7&=(\chi_{16})_N,&\theta_8&=\frac{1}{2}(\chi_{20})_N.\end{align*}

In particular, $N$ is non-abelian. It can be cheked that $N=\ze{A}\times B$.
\end{example}

\begin{example}\label{example2}
Let $G$ be the automorphism group of the dihedral group of order $16$. Its  character table is
the $11\times 11$ matrix appearing in Table~\ref{tabla21}. Let $N$ be the intersection of the kernels of $\chi_2$ and
$\chi_6$, so $N$ is a normal subgroup of $G$ which is the union of the
conjugacy classes $1A, 2B, 4A$ and $2D$, and its order is $8$.

\setlength{\arraycolsep}{0.7pt}

\begin{table}
\begin{footnotesize}
$$\begin{array}{cccccccccccc}
& 1A& 2A& 2B& 2C& 4A& 2D& 8A& 4B& 2E& 4C& 8B \\

&  1& 4& 4& 2& 2& 1& 4& 4& 4& 2& 4 \\

&&&&&&&&&&&\\

\chi_1&1& 1& 1& 1& 1& 1& 1& 1& 1& 1& 1 \\
\chi_2&1& -1& 1& 1& 1& 1& -1& -1& 1& 1& -1 \\
\chi_3&1& 1& -1& 1& 1& 1& -1& 1& -1& 1& -1 \\
\chi_4&1& -1& -1& 1& 1& 1& 1& -1& -1& 1& 1 \\
\chi_5&1& 1& 1& -1& 1& 1& 1& -1& -1& -1& -1 \\
\chi_6&1& -1& 1& -1& 1& 1& -1& 1& -1& -1& 1 \\
\chi_7&1& 1& -1& -1& 1& 1& -1& -1& 1& -1& 1 \\
\chi_8&1& -1& -1& -1& 1& 1& 1& 1& 1& -1& -1 \\
\chi_9&2& 0& 0& 2& -2& 2& 0& 0& 0& -2& 0 \\
\chi_{10}&2& 0& 0& -2& -2& 2& 0& 0& 0& 2& 0 \\
\chi_{11}&4& 0& 0& 0& 0& -4& 0& 0& 0& 0& 0 \\
\end{array}$$
\end{footnotesize}
\caption{Character table of the group $G$ in Example~\ref{example2}.}
\label{tabla21}
\end{table}

Then there are  $4$ equivalence classes in the equivalence relation respect to $N$, defined on $\irr{G}$, and so also $4$ minimal $G$-invariant characters of $N$, denoted $\widehat{\theta_i}$, $i=1,\dots,4$, with $\Omega=\{\theta_i\mid i=1,\dots,4\}$ a representative system of the orbits in the action by conjugation of $G$ on $\irr{N}$.
From the observation of proportional rows in the character table, when restricting to the described conjugacy classes, we obtain:
$$
\begin{array}{ll}
\irr{G|\theta_1}=\{\chi_1, \chi_2, \chi_5, \chi_6  \}, \quad & \irr{G|\theta_2}=\{\chi_3, \chi_4, \chi_7, \chi_8  \}, \\

\irr{G|\theta_3}=\{\chi_9, \chi_{10} \}, \quad & \irr{G|\theta_4}=\{\chi_{11}  \},
\end{array}$$

In particular, for each $i=1,\dots, 4$, we know the size of $ \irr{G|\theta_i}$. Also, whenever $\chi,\vartheta\in \irr{G|\theta_i}$, it holds that $$\chi_N=\vartheta_N=e_i\widehat{\theta_i}.$$

We may choose
$\{\chi_1, \chi_3, \chi_9,\chi_{11}\}$ a representative
system of the equivalence classes on $\irr{G}$, and construct the corresponding $G$-character table $\textsf{X}$ of $N$, which is the matrix of dimension $4\times 4$ in Table~\ref{tabla22}.

\setlength{\arraycolsep}{1.5pt}
\begin{table}
\begin{footnotesize}
$$\begin{array}{ccccc}
& 1A& 2B& 4A& 2D \\
& 1&4& 2& 1 \\
&&&&\\

\chi_1 & 1 & 1 & 1 & 1 \\

\chi_3 & 1 &-1 & 1 & 1   \\

\chi_9 & 2 &0 & -2 & 2   \\

\chi_{11} & 4 &0 & 0 & -4   \\
\end{array}$$
\end{footnotesize}
\caption{$G$-character table $\textsf{X}$ of $N$ in Example~\ref{example2}.}
\label{tabla22}
\end{table}

 Note also that $\textsf{D}=\text{diag}(1,4,2,1)$ is the associated diagonal matrix, whose diagonal entries are the sizes of the $G$-conjugacy classes that built $N$.

According to Corollary~\ref{relations}, we get in
this case that the numbers appearing in the right side of relations
$(A_i)$ and $(B_i)$, $1\leq i \leq 4$, are both $1,1,2,4$, respectively. From equations $(A_i)$ and $(B_i)$, $1\le i\le 4$, we deduce now:

For each $i=1,2$, $e_i=t_i=\theta_i(1)=1$, and for all $\chi\in \irr{G|\theta_i}$,  $\chi_N=\theta_i=\widehat{\theta_i}$.

For $i=3$, $e_3=\theta_3(1)=1$, $t_3=2$, and for all $\chi\in \irr{G|\theta_i}$,  $\chi_N=\theta_3+{\theta_3}^x=\widehat{\theta_3}$, for some $x\in G$.

However, for $i=4$, two possibilities appear:
either

$e_4=\theta_4(1)=1$, $t_4=4$, and $(\chi_{11})_N  = \theta_4+{\theta_4}^x+{\theta_4}^y+{\theta_4}^z=\widehat{\theta_4}$, for suitable $x,y,z\in G$, or

$e_4=\theta_4(1)=2$, $t_4=1$, and $(\chi_{11})_N  = 2\theta_4=2\widehat{\theta_4}.$

Actually, both cases are possible. The normal subgroup $M$, which
is the intersection of the kernels of $\chi_4$ and $\chi_8$ (and so $M$ is the union of the
conjugacy classes $1A, 8A, 4A$ and $2D$, and its order is $8$), happens to have the
same $G$-character table as $N$, by replacing $\chi_3$ by $\chi_2$, and $2B$ by $8A$, associated to the same diagonal matrix $\textsf{D}$ of $G$-conjugacy class sizes, so that the relations $(A_i)$, $(B_i)$, $1\le i\le 4$, are the same for both subgroups. Indeed, the    information derived for $N$ and $M$ is the same, with corresponding suitable irreducible characters of $G$. Considering, for instance, in addition the orders of elements in $N$ and $M$, one easily deduces que $M$ is abelian but $N$ is not, and so the first possibility for $i=4$ corresponds to $M$, while the second one corresponds to $N$.

\end{example}


\end{document}